\documentclass[12pt,twoside]{article} 
\usepackage{amsmath, amsthm,amssymb, amsfonts, bm} 
\usepackage{authblk} 
\usepackage[colorlinks=true,
            linkcolor=blue,     
            citecolor=red,      
            urlcolor=magenta    
           ]{hyperref}
\usepackage{fancyhdr}
\usepackage[a4paper,margin=1in,includeheadfoot]{geometry}

\hypersetup{
    pdftitle={The numerical radius of fractional powers of matrices}, 
    pdfauthor={Eman Aldabbas, Mohammad Sababheh },
    colorlinks=true,
    linkcolor=blue,
    urlcolor=blue
}
\newtheorem{theorem}{Theorem}[section]
\newtheorem{lemma}[theorem]{Lemma}

\newtheorem{remark}[theorem]{Remark}

\newtheorem{corollary}[theorem]{Corollary}
\newtheorem{proposition}[theorem]{Proposition}
\numberwithin{equation}{section}
\newtheorem{conjecture}[theorem]{Conjecture}
\providecommand{\keywords}[1]{\small \textbf{\textit{Keywords---}} #1}
\providecommand{\MSC}[1]{\small \textbf{\textit{MSC Classification 2020---}} #1}

\title{The numerical radius of fractional powers of matrices}

\author[1]{Eman Aldabbas \thanks{\href{e\_aldabbas@ju.edu.jo}{e\_aldabbas@ju.edu.jo, aldabbas@ualberta.ca}}}
\author[2,3]{Mohammad Sababheh \thanks{\href{sababheh@psut.edu.jo}{ sababheh@yahoo.com}}}
\affil[1]{Department of Mathematics, University of Jordan, Amman 11942, Jordan}
\affil[2]{Department of Basic Sciences, Princess Sumaya University for Technology\\ Amman 11941, Jordan}
\affil[3]{Department of Mathematics, Abdullah Al Salem University, Khaldiya, Kuwait}

\date{} 

\begin{document}

\maketitle

\abstract{Using integral representations of the fractional power of matrices, and the geometric intuition of sectorial matrices, we show that for any accretive-dissipative matrix $A$ and  any $t \in (0,1)$, the matrix \(A^t\) is accretive-dissipative, and that 
\[
\omega(A^t)\geq \omega^t(A) ,
\]
where \(\omega(\cdot)\) is the numerical radius. This inequality complements the well-known power inequality $\omega(A^k)\leq \omega^k(A)$, valid for any square matrix and positive integer power $k$. As an application, we prove that if $A$ is accretive, then the above fractional  inequality holds if $0<t<\frac{1}{2}$. Other consequences will be given too.}




\MSC{47A12, 15A60, 47B44}

\maketitle

\section{Introduction}
Throughout this paper, we use upper-case letters to denote square matrices of appropriate sizes. The zero matrix will be denoted by $O$, while the identity matrix by $I$. The algebra of all $n\times n$ matrices is denoted by $\mathcal{M}_n$. The unit sphere of \(\mathbb{C}^n\) is denoted by \(\mathcal{S}_1(\mathbb{C}^n)\).
For $A \in \mathcal{M}_n$, the numerical radius $\omega(A)$ and the spectral norm $\|A\|$ are defined, respectively, by
\[
\omega(A) = \max \{|\langle Ax, x \rangle| : x \in \mathcal{S}_1(\mathbb{C}^n)\}, 
\]
and
\[
\|A\| = \max \{\|Ax\| : x\in \mathcal{S}_1(\mathbb{C}^n)\}.
\]

The spectral norm is sub-multiplicative, in the sense that $\|AB\|\leq \|A\|\;\|B\|$ for all $A,B\in\mathcal{M}_n$. This immediately implies $\|A^k\|\leq \|A\|^k$ for any $A\in\mathcal{M}_n$ and any positive integer $k$. On the contrary, $\omega(\cdot)$ is not sub-multiplicative, yet it satisfies the power inequality \cite[Theorem 2.1-1]{Gustafson1997}
\begin{equation}\label{Eq_Intro_Power_k}
\omega(A^k)\leq \omega^k(A), A\in\mathcal{M}_n, k\in\mathbb{N},
\end{equation}
where the notation $\omega^k(A)$ refers to $(\omega(A))^k.$ While this has been a key property of the numerical radius,  we find no discussion of \eqref{Eq_Intro_Power_k} when the power is not an integer. It is the main goal of this paper to discuss possible relations between $\omega(A^t)$ and $\omega^t(A)$, when $0<t<1.$ 

To be able to discuss these fractional powers of $A$, we need a guarantee that this is well defined.

Given $A\in\mathcal{M}_n$, the numerical range of \(A\), denoted by $W(A)$, is defined as the image of \(\mathcal{S}_1(\mathbb{C}^n)\) under the quadratic form $x\mapsto \left<Ax,x\right>.$ That is,
\[W(A)=\left\{\left<Ax,x\right>: x\in \mathcal{S}_1(\mathbb{C}^n)\right\}.\]
If $W(A) \subset (0, \infty)$, we say that $A$ is positive definite and simply write $A > O$. It is well known that when $A > O$, then $\omega(A) = \|A\|$.\\
A useful property regarding the origin and the numerical range of \(A\) is the following:
\begin{proposition} \cite{Johnson1972}
    Let $A\in\mathcal{M}_n$ be such that $0\not\in W(A).$ Then there exists $\theta\in\mathbb{R}$ such that 
    \[W\left(e^{i\theta}A\right)\subseteq \{z\in\mathbb{C}: \Re z>0\},\]
    where $\Re z=\frac{z+z^*}{2}$ is the real part of $z$.
\end{proposition}
This motivates the definition of accretive matrices as those with numerical range in the open right half-plane. That is, we say that $A\in\mathcal{M}_n$ is accretive if \(W(A)\subseteq \{z\in\mathbb{C}: \Re z> 0\}\). The set of all accretive matrices in $\mathcal{M}_n$ is denoted by \(\Gamma_n\).  Using the Toeplitz decomposition, any matrix \(A\in \mathcal{M}_n\) can be written as \(A=H+iK\), where \[H=\Re(A)=\dfrac{A+A^*}{2} \text{ and } K=\Im(A)= \dfrac{A-A^*}{2i}.\] 
In this form, \(A\in\Gamma_n \Leftrightarrow H > O\). Moreover, \(A\) is called dissipative if \(\Im(A)>O\). If both \(\Re(A)\) and \(\Im(A)\) are positive definite (that is, \( H, K > O \)), then \( A \) is called accretive-dissipative. \\
Speaking of accretive matrices, we can describe sectorial matrices by those matrices whose numerical range lies in a certain sector.  A matrix \( A \in \mathcal{M}_n \) is called sectorial if its numerical range lies within a sector of the complex plane in the form
\[
S_\alpha := \{ z \in \mathbb{C} : |\Im z| \leq \tan \alpha \, \Re z \}
\]
for some angle \( 0 \leq \alpha < \frac{\pi}{2} \). 
 We can easily verify that \cite{Bedrani2022}:
 \[A\in\Gamma_n \Leftrightarrow W(A)\subset S_{\alpha}\;for\;some\;\alpha\in \left[0,\frac{\pi}{2}\right).\]
 
Some foundational work on the class of accretive matrices can be found in \cite{Johnson1972,Johnson1975,Kato1961_1,Kato1961_2,Kato1962}, making a substantial contribution to the field. Recently, the class of accretive matrices has gained a great deal of attention. We refer the reader to \cite{Aldabbas2025_1,Alemeh2025,Bedrani2021_1,Bedrani2021_2,Bedrani2021_3,Drury2015,Drury2014,Furuichi2024,Raissouli2017,Sababheh2024,Zhang2015} as  a list of such references. 

Now, to begin our investigation of $\omega(A^t)$, we recall the definition of such powers. Let \(A\in \mathcal{M}_n\) be such that \(W(A)\cap (-\infty,0]=\varnothing\), and let \(\Omega_A\) denote a contour in the resolvent  of \(A\) that winds once about each eigenvalue of $A$ and avoids \((-\infty,0]\). Then for any \(t\in(0,1)\), the prinicpal fractional power \(A^t\) can be defined via the Dunford integral as follows:
\begin{equation}\label{Dinford Integral}
A^t= \dfrac{1}{2\pi i}\displaystyle\int_{\Omega_A}\, z^t \Big(zI-A\Big)^{-1}\, dz,
\end{equation}
which coincides with the following formula (see \cite{Kato1961_1,Choi2019}): 
\begin{equation}\label{fractional power-definition}
A^t= \dfrac{\sin(t\pi)}{\pi} \displaystyle\int_{0}^{\infty}\, s^{t-1} A (sI+A)^{-1}\, ds.
\end{equation} 
It is important to remember that in \eqref{Dinford Integral}, and in what follows, $z^t$ refers to the principal power of $z$, and so is $A^t$.
One can use \eqref{fractional power-definition} to prove that if $A$ is accretive, then so is \(A^t\), for any \(t \in (0,1)\); see \cite[Lemma A6]{Kato1961_1}. An alternative proof was later given in \cite[Theorem 2.3]{Drury2015}. 

Before proceeding further, we state the following simple fractional consequence of \eqref{Eq_Intro_Power_k}.
\begin{proposition}\label{Prop_reci}
Let $A\in\Gamma_n$, and let $k\in\mathbb{N}.$ Then
\[\omega\left(A^{\frac{1}{k}}\right)\geq \omega^{\frac{1}{k}}(A).\]
\end{proposition}
\begin{proof}
Since $A$ is accretive, $A^{\frac{1}{k}}$ is a well-defined accretive matrix. Applying \eqref{Eq_Intro_Power_k} implies
\begin{align*}
\omega(A)=\omega\left(\left(A^{\frac{1}{k}}\right)^k\right)\leq \omega^k\left(A^{\frac{1}{k}}\right),
\end{align*}
which is equivalent to the desired statement.
\end{proof}
What Proposition \ref{Prop_reci} says is that \eqref{Eq_Intro_Power_k} is reversed if the power is the reciprocal of a natural number. Indeed, our simulations suggest that this is the case for any fractional power $t\in(0,1)$. That is, we have the following conjecture.

\begin{conjecture}\label{Conj_Main_1}
Let $A\in\Gamma_n$ and $0<t<1$. Then
\[\omega(A^t)\geq \omega^t(A).\]
\end{conjecture}
It is the sole goal of this paper to discuss this conjecture.

If \(A >O\), then clearly \(A^t>O\) for any \(t\in [0,1]\) and hence,
\begin{equation}\label{fractional power eq for positive matrices}
\omega(A^t)= \|A^t\|= \|A\|^t= \omega^t(A).    
\end{equation}
Therefore, the discussion of $\omega(A^t)$ when $A>O$ is trivial. 
While a proof of Conjecture \ref{Conj_Main_1} is not available yet, we will present some progress as follows. First, we will prove that for any accretive-dissipative matrix \(A\in\mathcal{M}_n\) and any \(t\in (0,1)\), the matrix \(A^t\) is indeed accretive-dissipative. This will be used to prove that Conjecture \ref{Conj_Main_1} is true if $A$ is accretive-dissipative. After that, we use this together with Proposition \ref{Prop_reci} to prove that Conjecture \ref{Conj_Main_1} is true if $0<t<\frac{1}{2}.$

The organization of this paper is as follows. In Section 2, we present some preliminary results and lemmas that are essential to prove our results. Section 3 contains the main results, where we study the behavior of the fractional powers of accretive-dissipative matrices and extend identity \eqref{fractional power eq for positive matrices} to this setting.
\section{Auxiliary Results}
In this section, we state several lemmas that are needed to prove our results.\\

We begin by presenting the following well-known result, which is a direct consequence of the spectral decomposition of Hermitian matrices.
\begin{lemma}\label{Fractioal Power ineq for the norm of A}
 Let \(A>O\) and let \(t\in [0,1]\). Then \[\Vert A\Vert^t= \Vert A^t\Vert.\]
\end{lemma}
Next, we present two lemmas of particular interest. Lemma \ref{accretivity of A^t} was independently proved by Kato and Drury. Lemma \ref{Fractional Power ineq for R(A)}, which is significant in its own right, was established later and, interestingly,  it implies the accretivity of fractional power \(A^t\).

\begin{lemma}(\cite[Lemma A1]{Kato1961_1},\cite[Corollary 2.4]{Drury2015})\label{accretivity of A^t}
    Let \(A\in\Gamma_n\) be  such that \(W(A)\subseteq S_{\alpha}\;for\;some\;\alpha\in \left[0,\frac{\pi}{2}\right)\), and let \(t\in [0,1]\). Then \(W(A^t)\subseteq S_{t\alpha}\).
\end{lemma}
\begin{lemma}\cite[Proposition 7.1]{Bedrani2021_1}\label{Fractional Power ineq for R(A)}
    Let \(A\in \Gamma_n\) and let \(t\in[0,1]\). Then \[\Re^t(A)\leq \Re(A^t).\]
\end{lemma}
While the class of accretive matrices is stable under taking inverses, the class of accretive-dissipative matrices is not, as the following lemma states.
\begin{lemma}\cite[Lemma 1]{Ikramov2004}\label{Inverse of A is not accretive-dissipative}
  Let \(A=H+iK\) be accretive-dissipative. Then \(A^{-1}=E+iF\), where \[E= \Big(H+KH^{-1}K\Big)^{-1}>O,\] 
  and 
  \[F= -\Big(K+HK^{-1}H\Big)^{-1}<O.\]  
  Hence, \(A^{-1}\) is not accretive-dissipative.
\end{lemma}

The last lemma is an equivalent identity for the numerical radius, and can be found in \cite{Haagerup1992}.
\begin{lemma}\label{w_theta_ident}
Let $A\in\mathcal{M}_n$. Then
\[\omega(A)=\sup_{\theta\in\mathbb{R}}\left\|\Re\left(e^{i\theta}A\right)\right\|.\]
\end{lemma}


\section{Main Results}
We begin this section by proving the following fact. Using Lemma \ref{w_theta_ident}, it can be easily seen that  $\omega(A)$ is attained at some angle \( \theta \in [0, 2\pi] \). The result below identifies the exact value of such an angle \( \theta \). To the best of the authors' knowledge, this specific proof does not appear in the existing literature. In what follows, $Arg(z)$ will refer to the principal argument of the complex number $z$. That is, $-\pi<Arg(z)\leq \pi.$ The goal of presenting this proof lies in the approach we are adopting to prove our main results, where this geometric approach will be crucial.

\begin{proposition}
    \label{N. rad is attained}
    Let \(A \in \mathcal{M}_n\) and let \(x_0 \in \mathcal{S}_1(\mathbb{C}^n)\) be such that \(\omega(A) = |\langle Ax_0, x_0 \rangle|\). If  \(\gamma = Arg(\langle Ax_0, x_0 \rangle)\), then 
    \[
    \omega(A) = \left\Vert \Re\left(e^{-i\gamma} A\right) \right\Vert.
    \]
\end{proposition}
\begin{proof}
Notice first that, by the compactness of \(\mathcal{S}_1(\mathbb{C}^n)\), we have \(\omega(A)= |\langle Ax_0,x_0\rangle|\) for some \(x_0\in \mathcal{S}_1(\mathbb{C}^n)\) .
Let \(\gamma\) be the principal argument of \(\langle Ax_0,x_0\rangle\). Then \[\langle Ax_0,x_0\rangle= |\langle Ax_0,x_0\rangle| e^{i\gamma} =\omega(A) e^{i\gamma}.\]
Now, noting that the matrix $ \Re\Big(e^{-i \gamma}\,A\Big)$ is Hermitian, and that for any matrix $X$, $\omega(\Re(X))\leq \omega(X)$, we deduce the following simple consequences:
\begin{align*}
    \Big\Vert \Re\Big(e^{-i \gamma}\,A\Big)\Big\Vert & = \omega\Big( \Re\Big(e^{-i \gamma}\,A\Big)\Big) \leq \omega(e^{-i \gamma}\,A) = \omega(A).
\end{align*}
On the other hand, 
\begin{align*}
    \Big\Vert \Re\Big(e^{-i \gamma}\,A\Big)\Big\Vert &  = \omega\Big( \Re\Big(e^{-i \gamma}\,A\Big)\Big)\\
    &= \displaystyle \sup_{||x||=1}\, \left| \langle \Re\Big(e^{-i \gamma}\,A\Big)x,x\rangle \right|\\
    & \geq \Big\vert \langle \Re\Big(e^{-i \gamma}\,A\Big)x_0,x_0\rangle \Big\vert\\
    &= \Big\vert \Re\Big(e^{-i \gamma}\,\langle Ax_0,x_0\rangle\Big)\Big\vert\\
    &= \Big\vert \Re\Big(e^{-i \gamma}\, e^{i\gamma} \omega(A)\Big)\Big\vert= \Big\vert \Re (\omega(A))\Big\vert\\
    &= \omega(A).
    \end{align*}
    This completes the proof.
\end{proof}
\begin{remark}\label{angular distance of w(A)}
    If \(A\) is accretive-dissipative, then \(W(A)\subseteq \{z\in \mathbb{C}: \Re(z), \Im(z)>0\}\). This means that $0< Arg\left<Ax,x\right><\frac{\pi}{2}$ for any $x\in\mathbb{C}^n.$  Consequently, Proposition \ref{N. rad is attained} implies that \(\omega(A)= \Big\Vert \Re\Big(e^{-i \gamma}\,A\Big)\Big\Vert\) for some \(\gamma\in (0,\pi/2)\).
\end{remark}
As we mentioned earlier, if $A\in\mathcal{M}_n$ is accretive, and $0<t<1$, then $A^t$ is accretive. In the following Proposition, we extend this observation to accretive-dissipative matrices.
\begin{proposition}\label{Fractional Power of Accretive-Dissipative}\label{A^t is accretive-dissipative}
    Let \(A\in \mathcal{M}_n\) be accretive-dissipative, and let $0<t<1$. Then \(A^t\) is accretive-dissipative.
\end{proposition}
\begin{proof}
Let \(A\in \mathcal{M}_n\) be accretive-dissipative. Then so is \(sI+A\) for \(s>0\). Consequently, by Lemma \ref{Inverse of A is not accretive-dissipative}, we have \(-\Im\Big((sI+A)^{-1}\Big)>O\).\\
Now, let \(t\in (0,1)\). Then, by Lemma \ref{accretivity of A^t}, \(\Re(A^t)>O\). So we only need to prove \(\Im(A^t)>O\). But, by \eqref{fractional power-definition}, we have
\[A^t= \dfrac{\sin(t\pi)}{\pi} \displaystyle\int_{0}^{\infty}\, s^{t-1} A (sI+A)^{-1}\, ds.\]
So, it is enough to prove that \(\Im\Big(A(sI+A)^{-1}\Big)\) is positive-definite. Since \(A(sI+A)^{-1}= I-s(sI+A)^{-1}\), it follows that
\[\Im\Big(A(sI+A)^{-1}\Big) = - \Im\Big((sI+A)^{-1}\Big)>O,\]
which completes the proof. 
\end{proof}
While the following observation follows from basic properties of matrix logarithmic functions, we present a proof using the Dunford integral, where we implement our geometric intuition, which is the key idea behind this work.
\begin{proposition}\label{Fractional Power of the rotational Matrix}
    Let \(A\in \Gamma_n\) and let \(\frac{-\pi}{2}<\theta<\frac{\pi}{2}\). Then, for any \(t\in (0,1)\),  \[\Big(e^{i\theta}A\Big)^t= e^{it \theta} A^t.\]
\end{proposition}

\begin{proof}
    Let \(A\in \Gamma_n\) and let \(\frac{-\pi}{2}<\theta<\frac{\pi}{2}\). Since $W(A)\subseteq S_{\alpha}$ for some $0<\alpha<\frac{\pi}{2}$,  it follows that \(W(e^{i\theta} A)\) avoids \((-\infty,0]\). Hence, using \eqref{Dinford Integral}, we can write 
    \begin{equation}\label{Dinford integral of the rotaion matrix}
    \Big(e^{i\theta} A\Big)^t = \dfrac{1}{2\pi i}\displaystyle\int_{\Omega_{e^{i\theta} A}}\, z^t \Big(zI-e^{i\theta}A\Big)^{-1}\, dz,\; 0<t<1.
    \end{equation}
    In this identity, \(\Omega_{e^{i\theta} A}\) refers to a rotation by $\theta$ of a contour $\Omega_{A}$ that lies completely in the right-half plane, and winds exactly once about each eigenvalue of $A$. Such assumption is justified by the fact that $A$ is accretive. Due to this, $\Omega_{e^{i\theta}A}$ avoids \((-\infty,0]\). 
    
    Now, 
    \begin{align*}
        \Big(e^{i\theta} A\Big)^t &= \dfrac{1}{2\pi i}\displaystyle\int_{\Omega_{e^{i\theta} A}}\, z^t \Big(zI-e^{i\theta}A\Big)^{-1}\, dz\\
        & = \dfrac{1}{2\pi i}\displaystyle\int_{\Omega_{e^{i\theta} A}}\,e^{-i\theta} z^t \Big(e^{-i\theta}zI-A\Big)^{-1}\, dz.
    \end{align*}
    Setting \(u=e^{-i\theta} z\) rotates the contour \(\Omega_{e^{i\theta}A}\) by an angle $-\theta$ to obtain $\Omega_{A}$ back. Now, since \(u\) is a complex number in the right half plane, then \(Arg(u) \in(\frac{-\pi}{2},\frac{\pi}{2})\) and consequently, \(-\pi<Arg(z)= Arg(u)+\theta <\pi\). This ensures the validity of the following factorization
    \[z^t= (e^{i\theta}u)^t=e^{it\theta}\, u^t.\] 
    Hence,
  \begin{align*}
      \Big(e^{-i\gamma}A\Big)^t & = \dfrac{1}{2\pi i}\displaystyle\int_{\Omega_{e^{i\theta} A}}\, z^t \Big(zI-e^{i\theta}A\Big)^{-1}\, dz\\
        & = \dfrac{1}{2\pi i}\displaystyle\int_{\Omega_{e^{i\theta} A}}\,e^{-i\theta} z^t \Big(e^{-i\theta}zI-A\Big)^{-1}\, dz\\
      &=\dfrac{1}{2\pi i}\displaystyle\int_{\Omega_A}\, (e^{i\theta}u)^t \Big(uI-A\Big)^{-1}\, du\\
      &= \dfrac{1}{2\pi i}\displaystyle\int_{\Omega_{u}}\,e^{it\theta} u^t \Big(uI-A\Big)^{-1}\, du\\
      &= e^{it\theta} \Big( \dfrac{1}{2\pi\, i}\displaystyle\int_{\Omega_A}\, u^t \Big(uI-A\Big)^{-1}\, du\Big)\\
      &= e^{it\theta} A^t.
  \end{align*}
  This completes the proof.
\end{proof}

Now, we are ready to prove  our main result, where we address Conjecture \ref{Conj_Main_1} for accretive-dissipative matrices.
\begin{theorem}\label{w(Fractional A)}
    Let \(A\in \mathcal{M}_n\) be accretive-dissipative and let \(t\in (0,1)\). Then \[\omega(A^t) \geq \omega^t(A) .\]
\end{theorem}

\begin{proof}
   Let \(x_0\in \mathcal{S}_1(\mathbb{C}^n)\) be such that \(\omega(A)= |\langle Ax_0,x_0\rangle|\), and let \(\gamma=Arg\langle Ax_0,x_0\rangle\). Since \(\Re\Big(e^{-i\gamma}A\Big)= \cos \gamma \Re(A)+ \sin\gamma \Im(A),\) and $0<\gamma<\frac{\pi}{2},$ it follows that \(\Re\Big(e^{-i\gamma}A\Big)>O,\) and hence \(e^{-i\gamma} A\) is accretive. Now,
   \begin{align*}
       \omega^t(A)&= \Big\Vert \Re\Big(e^{-i\gamma} A \Big)\Big\Vert^t \quad (\text{by Proposition \ref{N. rad is attained}})\\
       &= \Big\Vert \Re^t\Big(e^{-i\gamma} A\Big)\Big\Vert \quad (\text{by Lemma \ref{Fractioal Power ineq for the norm of A}})\\
        & \leq \Big\Vert \Re \Big((e^{-i\gamma } A)^t\Big)\Big\Vert
        \quad (\text{by Lemma \ref{Fractional Power ineq for R(A)}})\\
        & = \Big\Vert \Re \Big(e^{-it\gamma } A^t\Big)\Big\Vert
        \quad (\text{by Proposition \ref{Fractional Power of the rotational Matrix}})\\
        &\leq \displaystyle\sup_{\theta \in \mathbb{R}}\,\Big\Vert \Re\Big(e^{i\theta} A^t \Big)\Big\Vert\\
        &= \omega(A^t) \quad({\text{by Lemma \ref{w_theta_ident}}}).
   \end{align*}
   This completes the proof.
\end{proof}
Now for accretive matrices, we are able to prove the following partial affirmative answer to our conjecture. Notice that Theorem \ref{w(Fractional A)} plays a vital role in the proof of this result.
\begin{corollary}\label{cor_accretive_1/2}
Let $A\in\Gamma_n$ and let $0<t<\frac{1}{2}$. Then
\[\omega(A^t)\geq\omega^t(A).\]
\end{corollary}
\begin{proof}
If $A\in\Gamma_n$, it follows by Lemma \ref{accretivity of A^t} that $W(A^{\frac{1}{2}})\subset S_{\pi/4}$. Consequently, the matrix \(e^{i \frac{\pi}{4}} A^{\frac{1}{2}}\) is accretive-dissipative. For convenience, let $r=2t$ so that $0<r<1.$ Applying Theorem \ref{w(Fractional A)} on $ e^{i\frac{\pi}{4}} A^{\frac{1}{2}}$, and noting Proposition \ref{Prop_reci}, we obtain
\begin{align*}
\omega(A^t)=\omega\left(\left(A^{\frac{1}{2}}\right)^r\right) &= \omega\left( e^{i (r\,\frac{\pi}{4})}\left(A^{\frac{1}{2}}\right)^r\right)\\ & = \omega\left( \left(e^{i\,\frac{\pi}{4}}A^{\frac{1}{2}}\right)^r\right) \\ & \geq \omega^r\left(e^{i\,\frac{\pi}{4}}\,A^{\frac{1}{2}}\right)\\
& = \omega^r\left(A^{\frac{1}{2}}\right)\\ & \geq \omega^{\frac{r}{2}}\left(A\right)=\omega^t(A).\\
\end{align*}
This completes the proof.
\end{proof}
At this point, we are in a position to discuss Conjecture \ref{Conj_Main_1} when $t>1$. For this, a careful treatment of $A^t$ should be done. Notice that if $A$ is accretive, and $t>1$, then $A^t$ is not necessarily accretive. However, if for some  $k\geq 2$, the matrix $A^k$ is accretive, it follows that $A^m$ is accretive for all $2\leq m\leq k$. This is due to Lemma \ref{accretivity of A^t}, and the fact that if $2\leq m\leq k$, then $m=\alpha k$ for some $\frac{2}{k}\leq \alpha\leq 1.$ This observation together with Corollary \ref{cor_accretive_1/2} implies the following complementary result. In this result, $m$ and $k$ are not necessarily integers.

\begin{corollary}
Let $A$ be accretive, and assume that $A^k$ is accretive for some positive number $k\geq 2$. Then 
\[\omega^m(A)\geq \omega(A^m) ,\]
for $2\leq m\leq k.$
\end{corollary}
\begin{proof}
Notice that since $A^k$ is accretive, and $m\leq k$, $A^m$ is necessarily accretive. Therefore, we may apply Corollary \ref{cor_accretive_1/2} with $A^{m}$ instead of $A$, and with $\frac{1}{m}$ instead of $t$, noting that $\frac{1}{m}\leq \frac{1}{2}.$ This implies
\begin{align*}
\omega(A)&=\omega\left(\left(A^m\right)^{\frac{1}{m}}\right)\geq \omega^{\frac{1}{m}}(A^m),
\end{align*}
which is the desired result.
\end{proof}

\section*{Acknowledgment}
The authors would like to thank Professor Stephen Drury of McGill University for his valuable and insightful discussion on the problem. 
\subsection*{Declarations}
\begin{itemize}
\item {\bf{Availability of data and materials}}: Not applicable.
\item {\bf{Competing interests}}: The authors declare that they have no competing interests.
\item {\bf{Funding}}: Not applicable.
\item {\bf{Authors' contributions}}: Authors declare that they have contributed equally to this paper. All authors have read and approved this version.
\end{itemize}


\end{document}